\documentclass[a4paper,12pt]{article}
\usepackage[utf8]{inputenc}
\usepackage{amsmath}
\usepackage{amsthm,amssymb,amsfonts}
\usepackage{graphicx}

\usepackage{array,bm}
\usepackage{color}
\usepackage{verbatim} %para poder comentat todo lo que hay entre \begin{comment} y \end{comment}

\newtheorem{theorem}{Theorem}% [section]

\newtheorem{corollary}[theorem]{Corollary}
\newtheorem{proposition}[theorem]{Proposition}

\newtheorem{example}[theorem]{Example}
\newtheorem{lemma}[theorem]{Lemma}

\title{Norm estimates for selfadjoint Toeplitz operators on the Fock space}
\author{Antonio Galbis}

\begin{document}
\maketitle
\abstract{An estimate for the norm of selfadjoint Toeplitz operators with a radial, bounded and integrable symbol is obtained. This emphasizes the fact that the norm of such operator is strictly less than the supremum norm of the symbol. Consequences for time-frequency localization operators are also given.}

\section{Introduction}
The Bargmann-Fock space ${\mathcal F}^2({\mathbb C})$ is the Hilbert space consisting of those analytic functions $f\in H({\mathbb C})$ such that
$$
\|f\|_{{\mathcal F}}^2 = \int_{{\mathbb C}}|f(z)|^2 e^{-\pi |z|^2}\ dA(z) < +\infty,
$$ where $dA(z)$ denotes the Lebesgue measure. ${\mathcal F}^2({\mathbb C})$ admits a reproducing kernel $K_w(z) = e^{\pi \overline{w}z},$ which means that $$f(w) = \langle f, K_w\rangle,\ \ f\in {\mathcal F}^2({\mathbb C}).$$ The normalized monomials 
$$
e_n(z) = \left(\frac{\pi^n}{n!}\right)^{\frac{1}{2}} z^n,\ \ n\geq 0,
$$ form an orthonormal basis. 
For a fixed $a\in {\mathbb C}$ the translation operator
$$
T_a:{\mathcal F}^2({\mathbb C})\to {\mathcal F}^2({\mathbb C}),\  \left(T_a f\right)(z) = f(z-a)e^{-\frac{\pi}{2} |a|^2 + \pi z\overline{a}},
$$ is an isometry (see \cite{zhu}). We denote $d\lambda(z) = e^{-\pi |z|^2} dA(z),$ so ${\mathcal F}^2({\mathbb C})$ is a closed subspace of $L^2({\mathbb C}, d\lambda).$ The orthogonal projection 
$$
P:L^2({\mathbb C}, d\lambda)\to {\mathcal F}^2({\mathbb C})$$ is the integral operator 
$$
\left(Pf\right)(z) = \int_{{\mathbb C}} f(w)K_w(z)\ d\lambda(w).$$ For a measurable and bounded function $F$ on ${\mathbb C}$ the Toeplitz operator with symbol $F$ is defined as 
$$
T_F(f)(z) = P(Ff)(z) = \int_{{\mathbb C}} F(w)f(w)K_w(z)\ d\lambda(w).$$ The systematic study of Toeplitz operators on the Fock space started in \cite{bc1,bc2}. Since then it has been a very active research area. We refer to \cite[Chapter 6]{zhu-book}, where boundedness and membership in the Schatten classes is discussed.

It is obvious that $$T_F:{\mathcal F}^2({\mathbb C})\to {\mathcal F}^2({\mathbb C})$$ is a bounded operator and 
$$
\|T_F(f)\| \leq \|Ff\|_{L^2({\mathbb C}, d\lambda)} \leq \|F\|_\infty\cdot \|f\|.$$ In particular, $\|T_F\|\leq 1$ whenever $\|F\|_\infty\leq 1.$ If moreover $T_F$ is compact, which happens for instance when $F\in L^1({\mathbb C}),$ then $\|T_F\|$ is strictly less than 1 but, as far as we know, no precise estimate for the norm is known. The main result of the paper gives a bound for $\|T_F\|$ in the case that the symbol $F$ is radial, real-valued, and satisfies some integrability condition. For Toeplitz operators with radial symbols we refer to \cite{gv}. 
\par\medskip
Besides Toeplitz operators on the Fock space we consider time-frequency localization operators with Gaussian window, also known as anti-Wick operators. They where introduced by Daubechies \cite{d} as filters in signal analysis and can be obtained from Toeplitz operators on the Fock space after applying Bargmann transform.

\section{Toeplitz operators on the Fock space}

The Toeplitz operator defined by a real valued symbol $F$ is self-adjoint. This is immediate form the identity

$$
\langle T_F(f), g\rangle = \int_{{\mathbb C}} F(z) f(z)\overline{g(z)}d\lambda(z)$$ for all $f,g\in {\mathcal F}^2({\mathbb C}).$ In this case we have
$$
\|T_F\| = \sup_{\|f\|=1}\left|\langle T_F(f), f\rangle\right| \leq \sup_{\|f\|=1}\int_{{\mathbb C}} |F(z)|\cdot |f(z)|^2d\lambda(z).$$

A symbol $F$ is said to be radial with respect to $a\in {\mathbb C}$ if $F(z) = g(|z-a|)$ for some bounded and measurable function $g$ on $[0, +\infty).$ The main result of the paper is as follows.

\begin{theorem}\label{th:main} Let $F\in L^1({\mathbb C})\cap L^\infty({\mathbb C})$ be a real-valued and radial symbol with respect to $a\in {\mathbb C}.$ Then 
	$$
	\|T_F\|\leq \|F\|_\infty \Big(1 - \exp\Big(-\frac{\|F\|_1}{\|F\|_\infty}\Big)\Big).$$
\end{theorem}
\par\medskip
For the proof we will need some auxiliary results. First we observe that for $|F(z)| = g(|z|)$ and $\displaystyle f = \sum_{n=0}^\infty b_n e_n$ we have, after changing to polar coordinates, 
$$
\begin{array}{*2{>{\displaystyle}l}}
\int_{{\mathbb C}} |F(z)|\cdot |f(z)|^2d\lambda(z) & = \sum_{n=0}^\infty |b_n|^2\int_{{\mathbb C}}g(|z|)|e_n(z)|^2\ d\lambda(z) \\ & \\ & = \sum_{n=0}^\infty |b_n|^2 2\pi \int_0^\infty g(r)\frac{r^{2n+1}}{n!}e^{-\pi r^2}\ dr \\ & \\ & = \sum_{n=0}^\infty |b_n|^2 \int_0^\infty g\Big(\sqrt{\frac{t}{\pi}}\Big)\frac{t^n}{n!} e^{-t}\ dt.
\end{array}
$$
The $d$-dimensional Lebesgue measure of a set $\Omega\subset {\mathbb R}^d$ is denoted $|\Omega|$ both for $d=1$ and $d=2.$

\begin{lemma}\label{lem:firstbound_integral} Let $I\subset [0, +\infty)$ be a measurable set with finite Lebesgue measure. Then
	$$
	\frac{1}{n!}\int_I s^ne^{-s}\ ds \leq 1-e^{-|I|}.$$
\end{lemma}
\begin{proof}
(a) We first assume that $I$ is a finite union of bounded intervals. Let $t_n > 0$ the absolute maximum of $h(s) = \frac{s^n}{n!}e^{-s}.$ Then $h$ increases on $[0,t_n]$ and decreases on $[t_n, +\infty).$ We consider $a\leq t_n\leq b$ such that 
	$$
	t_n - a = |I\cap [0, t_n]|\ ,\ b-t_n = |I\cap [t_n, +\infty)|.$$ Then
	$$
	\begin{array}{*2{>{\displaystyle}l}}
		\frac{1}{n!}\int_I s^n e^{-s}\ ds & \leq \int_a^b h(s)\ ds = \frac{e^{-a}}{n!}\int_0^{b-a}(t+a)^n e^{-t}\ dt \\ & \\ & 
		= \sum_{k=0}^n \binom{n}{k}\frac{a^{n-k}}{n!}e^{-a}\int_0^{|I|} t^k e^{-t}\ dt \\ & \\ & = 
		\sum_{k=0}^n \frac{a^{n -k}}{(n -k)!}e^{-a}\frac{1}{k!}\int_0^{|I|} t^k e^{-t}\ dt\\ & \\ & \leq \sup_{0\leq k\leq 
			n}\frac{1}{k!}\int_0^{|I|} t^k e^{-t}\ dt = \int_0^{|I|}e^{-t}\ dt.
	\end{array}
	$$ For the last identity observe that 
	$$
	\frac{1}{k!}\int_0^{s} t^k e^{-t}\ dt = 1 - e^{-s}\sum_{j=0}^k \frac{s^j}{j!}.
	$$	
(b) For a general measurable set $I$ with finite measure the conclusion follows from part (a) and the fact that for every $\varepsilon > 0$ there is a set $J,$ finite union of bounded intervals, with the property that 
$$
|J\setminus I| + |I\setminus J| \leq \varepsilon.$$	
\end{proof}	

\begin{lemma}
	Let $(I_k)_{k=1}^N$ be disjoint sets with finite measure and $0\leq \varepsilon_k\leq 1$ for every $1\leq k\leq N.$ Then, for every $p\in {\mathbb N}_0$ we have
	$$
	\sum_{k=1}^N \varepsilon_k\int_{I_k}\frac{t^p}{p!}e^{-t}\ dt \leq 1 - \exp\big(-\sum_{k=1}^N \varepsilon_k |I_k|\big).$$
\end{lemma}
\begin{proof}
	We denote $n$ the number of indexes $k$ such that $0 < \varepsilon_k < 1$ and we proceed by induction on $n.$ For $n = 0$ this is the content of lemma \ref{lem:firstbound_integral}. Let us now assume $n = 1.$ Let $1\leq j\leq N$ be the coordinate with the property that $0 < \varepsilon_j < $ and  check that 
	$$
	\psi(\varepsilon):= \sum_{k\neq j}\int_{I_k}\frac{t^p}{p!}e^{-t}\ dt + \varepsilon \int_{I_j}\frac{t^p}{p!}e^{-t}\ dt + \exp\big(-\sum_{k\neq j}|I_k| - \varepsilon |I_j|\big) \leq 1$$ for every $0 \leq \varepsilon \leq 1.$ In fact, $\psi(0) \leq 1$ and $\psi(1)\leq 1$ follow from Lemma \ref{lem:firstbound_integral}. Moreover, the critical point $\varepsilon_0$ of $\psi$ satisfies 
	$$
	\int_{I_j}\frac{t^p}{p!}e^{-t}\ dt = |I_j|\exp\big(-\sum_{k\neq j}|I_k| - \varepsilon_0 |I_j|\big).$$ Hence
	$$
	\begin{array}{*2{>{\displaystyle}l}}
		\psi(\varepsilon_0) & = \sum_{k\neq j}\int_{I_k}\frac{t^p}{p!}e^{-t}\ dt + \varepsilon_0|I_j|\exp\big(-\sum_{k\neq j}|I_k| - \varepsilon_0 |I_j|\big) \\ & \\ & + \exp\big(-\sum_{k\neq j}|I_k| - \varepsilon |I_j|\big)\\ & \\ & = \sum_{k\neq j}\int_{I_k}\frac{t^p}{p!}e^{-t}\ dt + \big(1 + \varepsilon_0|I_j|\big)\exp\big(-\sum_{k\neq j}|I_k| - \varepsilon |I_j|\big).
	\end{array}
	$$ Since 
	$$
	1 + \varepsilon_0|I_j| \leq \exp\big(\varepsilon_0 |I_j|\big)$$ we conclude 
	$$
	\psi(\varepsilon_0) \leq \sum_{k\neq j}\int_{I_k}\frac{t^p}{p!}e^{-t}\ dt + \exp\big(-\sum_{k\neq j}|I_k|\big) \leq 1.$$ 
	\par\medskip
	Let us assume that the Lemma holds for $n = \ell$ ($0\leq \ell < N$) and let $n = \ell +1.$  We consider the function $\psi:[0,1]^{\ell+1}\to {\mathbb R}$ defined by 
	$$
	\psi({\bm\varepsilon}):= \sum_{k=1}^{\ell+1}\varepsilon_k \int_{I_k}\frac{t^p}{p!}e^{-t}\ dt + \sum_j\int_{J_j}\frac{t^p}{p!}e^{-t}\ dt + \exp\big(-\sum_k\varepsilon_k |I_k| - \sum_j |J_j|\big)$$ for ${\bm \varepsilon} = (\varepsilon_1, \ldots, \varepsilon_{\ell+1}).$ The induction hypothesis means that $\psi({\bm \varepsilon}) \leq 1$ whenever ${\bm \varepsilon}$ is in the boundary of $[0,1]^{\ell+1}.$ The lemma is proved after checking that $\psi({\bm \varepsilon}_0) \leq 1,$ where ${\bm \varepsilon_0}$ is a critical point of $\psi.$ Proceeding as before,
	$$
	\begin{array}{*2{>{\displaystyle}l}}
		\psi({\bm \varepsilon}_0) & = \big(\sum_{k=1}^{\ell+1}\varepsilon_k |I_k| + 1\big)e^{-\sum_k \varepsilon_k |I_k|} e^{-\sum_j |J_j|} + \sum_j \int_{J_j}\frac{t^p}{p!}e^{-t}\ dt \\ & \\ & \leq \exp\big(-\sum_j |J_j|\big) + \sum_j \int_{J_j}\frac{t^p}{p!}e^{-t}\ dt \leq 1.
	\end{array}	
	$$
\end{proof}

{\it Proof of Theorem \ref{th:main}:} We first assume $a = 0,$ that is, $F$ is radial. After replacing $F$ by $G = \frac{F}{\|F\|_\infty}$ if necessary we can assume that $\|F\|_\infty = 1.$ Since $F$ is radial we have $F(z) = g\big(|z|\big).$ We aim to prove that 
$$
\int_{\mathbb C}\left|g\big(|z|\big)\right|\cdot \left|f(z)\right|^2 e^{-\pi|z|^2}\ dA(z) \leq 1 - \exp\Big(-2\pi\int_0^\infty r\left|g(r)\right|\ dr\Big)$$ for every entire function $f(z) = \displaystyle \sum_{n=0}^\infty b_p e_p$ such that $\displaystyle\sum_{p=0}^\infty |b_p|^2 = 1.$ We have
$$
\int_{\mathbb C}\left|g\big(|z|\big)\right|\cdot \left|F(z)\right|^2 e^{-\pi|z|^2}\ dA(z) = \sum_{p=0}^\infty |b_p|^2\int_0^\infty \left|g\Big(\sqrt{\frac{t}{\pi}}\Big)\right|\cdot \frac{t^p}{p!}e^{-t}\ dt.$$
Let us first assume 
\begin{equation}\label{eq:stepfunction}
	g = \sum_{k=1}^N \varepsilon_k \chi_{I_k},\ \left|\varepsilon_k\right|\leq 1,
\end{equation} where $(I_k)_{k=1}^N$ are disjoint intervals. Then, Lemma \ref{lem:firstbound_integral} gives 
$$
\begin{array}{*2{>{\displaystyle}l}}
	\sum_{p=0}^\infty |b_p|^2\int_0^\infty \left|g\Big(\sqrt{\frac{t}{\pi}}\Big)\right|\cdot \frac{t^p}{p!}e^{-t}\ dt & \leq 1 -	\exp\big(-\sum_{k=1}^N |\varepsilon_k| |J_k|\big)\\ & \\ & = 1 - \exp\big(-2\pi\int_0^\infty r|g(r)|\ dr\big)\\ & \\ & = 1 - \exp\big(-\|F\|_1\big).\end{array}$$ We used $J_k = \left\{t:\ \sqrt{\frac{t}{\pi}}\in I_k\right\}$ and 
$\displaystyle|J_k| = 2\pi\int_{I_k}r dr.$ Theorem \ref{th:main} is proved for $g$ as in (\ref{eq:stepfunction}). Let us now assume that $\|g\|_\infty \leq 1$ and $g\in L^1({\mathbb R}^+,rdr)\cap L^\infty({\mathbb R}^+).$ Then there is sequence $(g_n)_n$ of step functions as in (\ref{eq:stepfunction}) such that $$\lim_{n\to\infty}\int_0^\infty |g_n(r)-g(r)|\ r dr = 0.$$ We put $F_n(z):= g_n(|z|).$ Since 
$$
\lim_{n\to \infty}\|T_F - T_{F_n}\| \leq \lim_{n\to \infty}\|F_n - F\|_1 = 0,$$ we finally conclude
$$
\|T_F\|\leq 1 - \exp\big(-\|F\|_1\big).$$
\par\medskip
In the case $a\neq 0,$ the identity 
$$
\int_{{\mathbb C}} g\big(|z-a|\big)|f(z)|^2 d\lambda(z) = \int_{{\mathbb C}} g\big(|u|\big)\Big(T_{-a}f\Big)(u)d\lambda(u)$$ and the fact that $T_{-a}$ is an isometry gives the conclusion. We can also argue from the fact that $T_{-a}\circ T_F = T_G\circ T_{-a},$ where $G(z) = g\big(|z|\big).$\ \ $\Box$
\par\medskip
In particular, if $\Omega\subset {\mathbb C}$ presents radial symmetry with respect to some point then
\begin{equation}\label{eq:conjecture-fock}
\int_\Omega \left|f(z)\right|^2 d\lambda(z) \leq \Big(1 - e^{-|\Omega|}\Big)\cdot\int_{{\mathbb C}}\left|f(z)\right|^2 d\lambda(z)
\end{equation} for every $f\in {\mathcal F}^2({\mathbb C}).$
\par\medskip
The question arises whether inequality (\ref{eq:conjecture-fock}) holds for every subset $\Omega.$ This is related to a conjecture by Abreu and Speckbacher in \cite{abreu} (see the next section). We do not have an answer to this question except for monomials or its translates. 

\begin{example}\label{ex:reproducing-kernel} Let $k_w = e^{-\frac{\pi}{2}|w|^2}K_w$ be the normalized reproducing kernel of ${\mathcal F}^2({\mathbb C}).$ Then, for every set $\Omega\subset {\mathbb C}$ with finite measure we have
	$$
	\int_\Omega \left|k_w(z)\right|^2 d\lambda(z) \leq 1 - e^{-|\Omega|}.$$
\end{example}
\begin{proof}
	In fact, $k_w = T_w\left(e_0\right).$ Hence
	$$
	\int_\Omega |k_w(z)|^2 d\lambda(z) = \int_{\Omega-w} d\lambda(z) 
	$$ and the conclusion follows from the fact that the last integral attains its maximum when $\Omega$ is a disc centered at $w$ (see \cite[Proposition 9]{abreu}). 
\end{proof}
\par\medskip
It is easy to check that when $\Omega$ is a disc centered at point $\omega$ the inequality in Example \ref{ex:reproducing-kernel} is an identity.

\begin{proposition}\label{prop} Let $\Omega \subset {\mathbb R}^2$ be a set with finite measure. Then, for every $n\in {\mathbb N}$ and $a\in {\mathbb C},$
	$$
	\int_\Omega \left|T_a(e_n)(z)\right|^2 d\lambda(z)\leq 1-e^{-|\Omega|}.$$
\end{proposition}
\begin{proof} Since
	$$
	\int_\Omega \left|T_a(e_n)(z)\right|^2 d\lambda(z) = \int_{\Omega-a} \left|e_n(z)\right|^2 d\lambda(z)$$ we can assume that $a = 0.$ For every $\theta\in [0,2\pi]$ we denote $$
	\Omega_\theta = \left\{r \geq 0:\ re^{i\theta}\in\Omega\right\}.$$ Then 
	$$
	\begin{array}{*2{>{\displaystyle}l}}
		\int_\Omega \left|e_n(z)\right|^2 d\lambda(z) & = \frac{\pi^n}{n!}\int_\Omega \left|z^n\right|^2 e^{-\pi |z|^2}\ dA(z) \\ & \\ & = \frac{\pi^n}{n!}\int_0^{2\pi}\Big(\int_{\Omega_\theta}r^{2n}e^{-\pi r^2}2\pi r\ dr\Big)	\frac{d\theta}{2\pi}\\ & \\ & = \int_0^{2\pi}\Big(\int_{I_\theta}\frac{t^n}{n!} e^{-t}\ dt\Big) \frac{d\theta}{2\pi},
	\end{array}
	$$ where
	$$
	I_\theta = \left\{t = \pi r^2:\ r\in\Omega_\theta\right\}.$$ Since $|\Omega| < \infty$ then a.e. $\theta\in [0,2\pi]$ we have
	$$
	|I_\theta| = 2\pi\int_{\Omega_\theta}r dr < +\infty.$$ Moreover, by Lemma \ref{lem:firstbound_integral}, 
	$$
	\int_0^{2\pi}\Big(\int_{I_\theta}\frac{t^n}{n!} e^{-t}\ dt\Big) \frac{d\theta}{2\pi} \leq \int_0^{2\pi} \Big(1-e^{-|I_\theta|}\Big)\frac{d\theta}{2\pi}.$$ Finally we consider the convex function $f(t) = e^{-t}-1$ and the probability measure $\frac{d\theta}{2\pi}$ and put $h(\theta) = |I_\theta|.$ Jensen's inequality gives 
	$$
	f\Big(\int_0^{2\pi}h(\theta)\frac{d\theta}{2\pi}\Big) \leq \int_0^{2\pi}f\left(h(\theta)\right)\frac{d\theta}{2\pi},$$ which means 
	$$
	\begin{array}{*2{>{\displaystyle}l}}
		\int_0^{2\pi} \Big(1-e^{-|I_\theta|}\Big)\frac{d\theta}{2\pi} & \leq 1 - \exp\Big(-\int_0^{2\pi}|I_\theta|\ \frac{d\theta}{2\pi}\Big) \\ & \\ & = 1 - \exp\Big(-\int_0^{2\pi}\Big(\int_{\Omega_\theta} r\ dr\Big)\ d\theta\Big) \\ & \\ & = 1 - e^{-|\Omega|}.
	\end{array}$$	
\end{proof}

We finish the section with some examples of sets $\Omega$ with infinite Lebesgue measure for which the Toeplitz operator with symbol $F=\chi_\Omega$ has norm as small as we want.

\begin{proposition}\label{prop:small-norm} For every $\varepsilon > 0$ there exists $\Omega$ with infinite Lebesgue measure such that 
		$$
		\int_\Omega |f(z)|^2 d\lambda(z) \leq \varepsilon \int_{\mathbb C} |f(z)|^2 d\lambda(z)
		$$ for every $f\in {\mathcal F}^2.$
\end{proposition}
\begin{proof}
	Let us consider arbitrary $f\in {\mathcal F}^2,$ $R > 0$ and $\Omega\subset {\mathbb C}.$ We denote $C_R = \int_0^R 2\pi r e^{-\pi 
		r^2}\ dr.$ From
	$$
	|f(0)|^2\leq \frac{1}{2\pi}\int_0^{2\pi}|f(re^{i\theta})|^2\ d\theta\ \ \forall r > 0,
	$$ we get
	$$
	\begin{array}{*2{>{\displaystyle}l}}
		|f(0)|^2\cdot C_R & \leq \int_0^R\left(\int_0^{2\pi}|f(re^{i\theta})|^2\ d\theta\right) re^{-\pi r^2}\ dr\\ & \\ & = \int_{D(0,R)}|f(w)|^2 
		e^{-\pi|w|^2}\ dA(w).
	\end{array}
	$$ Then, for every $z\in {\mathbb C},$
	$$
	\begin{array}{*2{>{\displaystyle}l}}
		|f(z)|^2 e^{-\pi|z|^2} & = \left|\left(T_{-z}f\right)(0)\right|^2 \leq C_R^{-1}\cdot \int_{D(0,R)}|\left(T_{-z}f\right)(w)|^2 e^{-\pi|w|^2} dA(w) \\ & \\ & = 
		C_R^{-1}\cdot \int_{D(z,R)}|f(w)|^2 e^{-\pi|w|^2}\ dA(w).
	\end{array}
	$$ Finally
	$$
	\begin{array}{*2{>{\displaystyle}l}}
		\int_\Omega |f(z)|^2 e^{-\pi|z|^2} dA(z) & \leq C_R^{-1}\cdot \int_\Omega\left(\int_{\mathbb C}\chi_{D(w,R)}(z)|f(w)|^2 e^{-\pi|w|^2}\ 
		dA(w)\right)\ dA(z) \\ & \\ & \leq C_R^{-1}\cdot \int_{\mathbb C}|f(w)|^2 e^{-\pi|w|^2}\varphi(w)\ dA(w),
	\end{array}
	$$ where 
	$$
	\varphi(w) = \int_\Omega \chi_{D(w,R)}(z)\ dA(z) = \left|\Omega \cap D(w,R)\right|.
	$$ Now, take $\Omega$ the union of infinitely many disks of Lebesgue measure $\delta = \varepsilon\cdot C_R,$ sufficiently separated so that 
	each disk $D(w, R)$ can only cut one of them. Then $\varphi(w) \leq \varepsilon$ for every $w\in {\mathbb C}$ and the conclusion follows.
\end{proof}

\section{Time-frequency localization operators}

For $F\in L^1({\mathbb C})$ we denote  by $H_F:L^2({\mathbb R})\to L^2({\mathbb R})$ the localization operator 
$$
H_F f = \int_{\mathbb C} F(z)\ \langle f, \pi(z)h_0\rangle\ \pi(z)h_0\ dA(z).$$ Here $h_0(t) = 2^{1/4} e^{-\pi t^2}$ is the Gaussian and $\pi(z)$ is the time-frequency shift, defined for $z = x+i\omega$ as 
$$
\big(\pi(z)f\big)(t) = e^{2\pi i \omega t}f(t-x),\ f\in L^2({\mathbb R}).$$ In case $F$ is the characteristic function of a set $\Omega$ we write $H_\Omega$ instead of $H_{\chi_\Omega}.$ We refer to \cite{cg} or \cite[Chapter 4]{cr} for general facts concerning localization operators.
\par\medskip
For $f,g\in L^2({\mathbb R}),$ the expression 
$$
\big(V_g f\big)(z) := \langle f, \pi(z)g\rangle$$ is the short time Fourier transform of $f$ with window $g,$ known as Gabor transform in the case where the window $g = h_0$ is the Gaussian. 
\par\medskip
If $F$ is real-valued then $H_F$ is a selfadjoint operator on $L^2({\mathbb R}),$ hence
$$
\|H_F\| = \sup_{\|f\|_2 = 1}\left|\langle H_F f,\ f\rangle\right| \leq \sup_{\|f\|_2 = 1}\int_{{\mathbb C}} |F(z)|\cdot\left|\left(V_{h_0} f\right)(z)\right|^2 dA(z).$$ There is a connection between localization operators and Toeplitz operators on the Fock space via de Bargmann transform. 
\par\medskip
The Bargmann transform is the surjective and unitary operator 
$$
{\mathcal B}:L^2({\mathbb R}) \to {\mathcal F}^2({\mathbb C})
$$ defined as
$$
\big({\mathcal B}f\big)(z) = 2^{1/4}\int_{{\mathbb R}}f(t) e^{2\pi t z -\pi t^2-\frac{\pi}{2}z^2}\ dt.
$$ It was introduced in \cite{bargmann} and has the important property that the Hermite functions are mapped into normalized analytic monomials. More precisely, ${\mathcal B}(h_n) = e_n,$ where $h_n$ is defined via the so called Rodrigues formula as 
$$
h_n(t) = \frac{2^{1/4}}{\sqrt{n!}}\left(\frac{-1}{2\sqrt{\pi}}\right)^n e^{\pi t^2}\frac{d^n}{dt^n}\left(e^{-2\pi t^2}\right),\ \ n\geq 0.
$$ Then $\left(h_n\right)_{n\geq 0}$ forms an orthonormal basis for $L^2({\mathbb R}).$ The Gabor transform of Hermite functions is well-known (see for instance \cite[Chapter 1.9]{folland}). In fact, for $z = x + i\xi,$ 
\begin{equation}\label{eq:STFT-hermite}
	\langle h_n, \pi(z)h_0\rangle = e^{-i\pi x\xi - \frac{\pi}{2}|z|^2}\sqrt{\frac{\pi^n}{n!}}\overline{z}^n.\end{equation} Since for $z = x + i\xi$ we have (\cite[3.4.1]{g})
$$
\big(V_{h_0} f\big)(x,-\xi) = e^{i\pi x\xi}\cdot \big({\mathcal B}f\big)(z)\cdot e^{-\frac{\pi|z|^2}{2}}
$$ then, for every $f\in L^2({\mathbb R})$ and $F\in L^1({\mathbb C})\cap L^\infty({\mathbb C})$ we obtain
$$
\int_{{\mathbb C}} |F(z)|\cdot\left|\left(V_{h_0} f\right)(z)\right|^2 dA(z) = \int_{{\mathbb C}} |F(z)|\cdot \left|\big(Bf\big)(z)\right|^2\ d\lambda(z).$$ Consequently, all the estimates in the previous section can be translated into estimates concerning localization operators. 
\par\medskip
Abreu, Speckbacher conjecture in \cite{abreu} that, among all the sets with a given measure, $\|H_\Omega\|$ attains its maximum when $\Omega$ is a disc, up to perturbations of Lebesgue measure zero. This turns out to be equivalent to the validity of inequality (\ref{eq:conjecture-fock}) for every function in the Fock space or, equivalently, to the fact that
$$
\|f\|_2^2 \leq e^{|\Omega|}\int_{{\mathbb C}\setminus\Omega} \left|\big(V_{h_0}f\big)(z)\right|^2 dA(z)\ \ \forall f\in L^2({\mathbb R}).$$ This should be compared with \cite[Theorem 4.1]{fg}. In this regard it is worth noting that Nazarov \cite{n} proved the existence of two absolute constants $A, B$ such that
$$
\|f\|_2^2 \leq Ae^{B\cdot|S|\cdot|\Sigma|}\Big(\int_{{\mathbb R}\setminus S} |f|^2 + \int_{{\mathbb R}\setminus\Sigma}|\widehat{f}|^2\Big)$$ for every $f\in L^2({\mathbb R})$ and for any pair $(S,\Sigma)$ of sets with finite measure.
\par\medskip
From Theorem \ref{th:main} and Proposition \ref{prop} we get the following.

\begin{corollary}
Let $F\in L^1({\mathbb C})\cap L^\infty({\mathbb C})$ be a real-valued and radial symbol with respect to $a\in {\mathbb C}.$ Then 
$$
\|H_F\|\leq \|F\|_\infty \Big(1 - \exp\Big(-\frac{\|F\|_1}{\|F\|_\infty}\Big)\Big).$$
\end{corollary}

\begin{corollary}\label{cor:localization-hermite}
Let $\Omega \subset {\mathbb R}^2$ be a set with finite measure. Then, for every $n\in {\mathbb N},$
$$
\left|\langle H_\Omega h_n,\ h_n\rangle\right| \leq 1-e^{-|\Omega|}.$$
\end{corollary}

We fix a non-zero window $g\in L^2({\mathbb R}).$ The modulation space $M^1({\mathbb R}),$ also known as Feichtinger algebra, is the set of tempered distributions $f\in {\mathcal S}^\prime({\mathbb R})$ such that 
$$
\|f\|_{M^1}:= \int_{{\mathbb C}}\left|\langle f, \pi(z)g\rangle\right| dA(z) < +\infty.$$ The use of different windows $g$ in the definition of $M^1({\mathbb R})$ yields the same spaces with equivalent norms. It is well known that $M^1({\mathbb R})$ is continuously included in $L^2({\mathbb R})$ and 
$$
\|f\|_2 = \|V_g f\|_ 2 \leq \|V_g f\|_1$$ whenever $f\in M^1({\mathbb R})$ and $\|g\|_2 = 1.$ See for instance \cite[3.2.1]{g} for the first identity.

\begin{proposition} 
Let $\Omega \subset {\mathbb R}^2$ be a set with finite measure. Then, for every $f\in M^1({\mathbb R})$ and $n\in {\mathbb N}_0$ we have
$$
\int_\Omega \left|\left(V_{h_0} f\right)(z)\right|^2 dA(z) \leq \|V_{h_n} f\|_1^2\cdot \big(1-e^{-|\Omega|}\big).$$
\end{proposition}
\begin{proof} It suffices to prove the proposition under the additional assumption that $\|V_{h_n}f\|_1 = 1.$ Fixed $n\in {\mathbb N}_0$ we consider the set
$$
B:=\left\{\pi(z)h_n:\ z\in {\mathbb C}\right\} \subset L^2({\mathbb R}).$$ Then 
$$
B^\circ := \left\{g\in L^2({\mathbb R}):\ \left|\langle g, \pi(z)h_n\rangle\right|\leq 1\right\} = \left\{g\in L^2({\mathbb R}):\ \|V_{h_n}g\|_\infty\leq 1\right\}.$$ We have
$$
\left|\langle f, g\rangle\right| = \left|\langle V_{h_n}f, V_{h_n}g\rangle\right| \leq \|V_{h_n}f\|_1\cdot \|V_{h_n}g\|_\infty \leq 1$$ for every $g\in B^\circ,$ which means that $f\in B^{\circ\circ}.$ According to the bipolar theorem, 
$$
f = L^2-\lim_{k\to \infty} f_k$$ where each $f_k$ is in the absolutely convex hull of $B.$ For each $k\in {\mathbb N}$ we can find scalars $(\alpha_j)_{j= 1}^N$ and points $(z_j)_{j=1}^N$ such that $f_k = \sum_{j=1}^N\alpha_j \pi(z_j)h_n$ and $\sum_{j=1}^N |\alpha_j| \leq 1.$ Then
$$
\begin{array}{*2{>{\displaystyle}l}}
\Big(\int_\Omega \left|\left(V_{h_0} f_k\right)(z)\right|^2 dA(z)\Big)^{\frac{1}{2}}  & = \Big(\int_\Omega \left|\langle f_k, \pi(z)\varphi\rangle\right|^2 dA(z)\Big)^{\frac{1}{2}} \\ & \\ & \leq \sum_{j=1}^N |\alpha_j|\Big(\int_\Omega \left|\langle \pi(z_j)h_n, \pi(z)\varphi\rangle\right|^2 dA(z)\Big)^{\frac{1}{2}} \\ & \\ & = \sum_{j=1}^N |\alpha_j|\Big(\int_\Omega \left|\langle h_n, \pi(z-z_j)\varphi\rangle\right|^2 dA(z)\Big)^{\frac{1}{2}}\\ & \\ & = \sum_{j=1}^N |\alpha_j|\Big(\int_{\Omega-z_j} \left|\langle h_n, \pi(z)\varphi\rangle\right|^2 dA(z)\Big)^{\frac{1}{2}}\\ & \\ & = \sum_{j=1}^N |\alpha_j|\left|\langle H_{\Omega-z_j}h_n, h_n\rangle\right|^{\frac{1}{2}} \leq \big(1-e^{-|\Omega|}\big)^{\frac{1}{2}}.
\end{array}
$$ Finally,
$$
\int_\Omega \left|\left(V_{h_0} f\right)(z)\right|^2 dA(z) = \lim_{k\to \infty}\int_\Omega \left|\left(V_{h_0} f_k\right)(z)\right|^2 dA(z) \leq 1-e^{-|\Omega|}.$$
\end{proof}

\par\medskip
The next result is a direct consequence of Proposition \ref{prop:small-norm} and should be compared with \cite[Proposition 8]{abreu}.
\begin{corollary} For every $\varepsilon > 0$ there exists $\Omega$ with infinite Lebesgue measure such that
		$$
		\|H_\Omega\|\leq \varepsilon.
		$$
\end{corollary}

\end{document}